\documentclass[12pt]{amsart}
\usepackage{amssymb}
\title[Radial Function Approximation]
{Rate of Convergence and Tractability of \\
the Radial Function Approximation Problem}
\author{Gregory E. Fasshauer}
\address{Department of Applied Mathematics, Illinois Institute of Technology, Room E1-208, 10 W. 32$^{\text{nd}}$ Street, Chicago, IL 60616}
\email{fasshauer@iit.edu}
\author{Fred J. Hickernell}
\address{Department of Applied Mathematics, Illinois Institute of Technology, Room E1-208, 10 W. 32$^{\text{nd}}$ Street, Chicago, IL 60616}
\email{hickernell@iit.edu}
\author{Henryk Wo\'zniakowski}
\address{Department of Computer Science, Columbia University,
New York, NY 10027, and Institute of Applied Mathematics, University of Warsaw,
ul. Banacha 2,  02-097 Warsaw, Poland}
\email{henryk@cs.columbia.edu}

\thanks{This research was supported in part
for the first two authors by NSF-DMS grant \#0713848, 
for the second author by Department of Energy grant \#SC0002100, and
for the third author by NSF-DMS grant \#0914345.}

\date{\today}

\date{May, 2010}

\newcommand{\Oh}{{\mathcal O}}
\newcommand{\cl}{{\mathcal L}}

\newcommand{\bi}{{\boldsymbol{i}}}
\newcommand{\bx}{{\boldsymbol{x}}}
\newcommand{\bgamma}{{\boldsymbol{\gamma}}}

\theoremstyle{plain}
\newtheorem{theorem}{Theorem}
\newtheorem{corollary}{Corollary}
\newtheorem{lemma}{Lemma}

\theoremstyle{definition}

\def\il{\left<}
\def\ir{\right>}

\def\e{\varepsilon }
\def\phi{\varphi }
\def\epsilon{\varepsilon }

\def\rho{\varrho }

\def\lstd{{\Lambda^{\rm std}}}
\def\lall{{\Lambda^{\rm all}}}
\newcount\refnum
\def\refx{\global\advance\refnum by 1 {\the\refnum . \ }}

\def\({\biggl( }
\def\){\biggr) }

\def\c{{\bf c}}

\newcommand{\tlambda}{\tilde{\lambda}}

\newcommand{\bc}{\boldsymbol{c}}
\newcommand{\bk}{\boldsymbol{k}}
\newcommand{\mK}{\mathsf{K}}
\newcommand{\bj}{\boldsymbol{j}}

\newcommand{\bt}{\boldsymbol{t}}
\newcommand{\by}{\boldsymbol{y}}
\newcommand{\bzero}{\boldsymbol{0}}
\newcommand{\tK}{\widetilde{K}}

\newcommand{\naturals}{\mathbb{N}}
\newcommand{\reals}{\mathbb{R}}

\def\e{\varepsilon}

\newcommand{\stdall}{\vartheta}
\newcommand{\absnorm}{\psi}
\newcommand{\tC}{\tilde{C}}

\begin{document}

\begin{abstract}
This article studies the problem of approximating functions
belonging to a Hilbert space $H_d$ with an isotropic or anisotropic
Gaussian reproducing kernel,
$$
K_d(\bx,\bt) = \exp\left(-\sum_{\ell=1}^d\gamma_\ell^2(x_\ell-t_\ell)^2\right)
\ \ \ \mbox{for all}\ \   \bx,\bt\in\reals^d.
$$
The isotropic case corresponds to using
the same shape parameters for all coordinates, namely
$\gamma_\ell=\gamma>0$ for all $\ell$,
whereas the anisotropic case corresponds to varying shape parameters
$\gamma_\ell$. We are especially interested in moderate to large $d$.
We consider two classes of algorithms:

(1) using finitely many arbitrary linear functionals,

(2) using only finitely many function values.

\noindent
The pertinent error criterion is the worst case of such an algorithm over
the unit ball in $H_d$, with the error for a single
function given by the $\cl_2$ norm also with a Gaussian weight.

Since the Gaussian kernel is analytic, the minimal
worst case errors of algorithms that use at most $n$
linear functionals or $n$ function values
vanish like $\Oh(n^{-p})$ as $n$ goes to infinity. Here,
$p$ can be arbitrarily large,
but the leading coefficient may depend on $d$ (Theorem \ref{thm1}).
On the other hand, if $d$ dependence is taken into account,
the convergence rate may be quite slow.  If the goal is to make
the error smaller than $Cn^{-p}$ for some $C$ \emph{independent of
  $d$} or \emph{polynomially dependent on $d$},
then this is possible for any choice of shape parameters
with the largest $p$ equal to $1/2$,
provided that arbitrary linear functional data is
available (Theorem \ref{thm2}).  If the sequence of shape
parameters $\gamma_\ell$
decays to zero like $\ell^{-\omega}$ as $\ell$
(and therefore also $d$) tends to $\infty$, then
the largest $p$ is roughly $\max(1/2,\omega)$ (Theorem \ref{thm3}).
If only function values are available, dimension-independent
convergence rates are somewhat worse (Theorems \ref{thm4} and
\ref{thm5}).

If the goal is to make the error smaller than $Cn^{-p}$ \emph{times
the initial ($n=0$) error}, then the corresponding $p$ is roughly
$\omega$. Therefore it is the same as before iff $\omega\ge1/2$ (Theorem
\ref{thm7} and Corollary \ref{cor2}).
In particular, for the isotropic case, when $\omega=0$,
the error does not even
decay polynomially with $n^{-1}$ (Theorem \ref{thm6}).
In summary, excellent dimension independent error decay
rates are only possible when the sequence of shape parameters decays
rapidly.
\end{abstract}

\maketitle

\section{Introduction}

Algorithms for function approximation based on symmetric,
positive definite kernels are important and
fundamental tools for
numerical computation \cite{Buh03,Fas07,SW06,Wen05},
statistical learning
\cite{BTA04,CZ07,HTF09,RW06,SS02,Ste99,SC08,Wah90},
and are often used in engineering applications \cite{FSK08}.
These algorithms go by a variety of names, including radial
basis function methods \cite{Buh03},
scattered data approximation \cite{Wen05},
meshfree methods \cite{Fas07}, (smoothing) splines \cite{Wah90},
kriging \cite{Ste99},
Gaussian process models \cite{RW06} and support vector machines \cite{SC08}.

In a typical application we are given noisy or noiseless scalar
or vector data. For simplicity, this article treats only
the noiseless scalar case in which the data is of the
form $y_i=f(\bx_i)$ or $y_i=L_i(f)$ for $i=1,\ldots, n$.
That is, a function $f$ is sampled at the locations $\{\bx_1, \ldots, \bx_n\}$,
usually referred to as the \emph{data sites} or the \emph{design},
or more generally we know the values of $n$ linear functionals $L_i$
on $f$. Here we assume that the domain of $f$ is a subset of
$\reals^d$.
One then chooses a symmetric, positive definite kernel
$K_d$ (see \eqref{Kassump} below for the specific requirements),
ideally such that $f\in H(K_d)$, where $H(K_d)$ is a reproducing kernel
Hilbert space with the reproducing kernel $K_d$.
Then it is a good idea to construct an approximation $S_n(f)$ to $f$
which has the minimal norm among all elements in $H(K_d)$ that
interpolate the data. This corresponds to the \emph{spline}
algorithm and requires the solution of an $n\times n$ system of linear equations.
While the spline algorithm is optimal in the sense explained in
Section \ref{funapproxsec} below, there still remains
the important questions of how fast $S_n(f)$ converges to $f$
as the number of data $n$ tends to infinity, and how to choose
the data sites or linear functionals  to
maximize the rate of convergence to $f$. Another question is to study
how the error bounds depend on $d$. The latter question is especially
important when $d$ is large.

The typical convergence rates
(see, e.g., \cite{Fas07,Wen05}) are of the form
$\mathcal{O}(n^{-p/d})$,
where $p$ denotes the smoothness of the kernel $K_d$, and the design is
chosen optimally.
Unfortunately, for a finite $p$,
this means that as the dimension increases,
these known convergence rates deteriorate dramatically.
Furthermore, the dimension dependence of the leading constant
in the big $\mathcal{O}$-term is usually not known in these estimates.

This article studies Hilbert spaces with reproducing kernels
$K_d:\reals^d\times \reals^d\to\reals$.
The kernel is called \emph{translation invariant}
or \emph{stationary} if $K(\bx,\bt) = \tK(\bx-\bt)$.
In particular, the kernel is \emph{radially symmetric}
or \emph{isotropic} if $K(\bx,\bt) = \kappa(\|\bx-\bt\|^2)$,
in which case the kernel is called a \emph{radial (basic) function}.

A kernel commonly used in practice,
and one which is studied here, is the isotropic Gaussian kernel:
\begin{subequations}\label{Gausskernelboth}
\begin{equation} \label{Gausskernel}
K_d(\bx,\bt) = e^{-\gamma^2 \|\bx-\bt\|^2} \ \ \ \mbox{for all}
\ \ \ \bx,\bt\in\reals^d ,
\end{equation}
where a positive $\gamma$ is called the \emph{shape parameter}.
This parameter functions as an inverse length scale.
Choosing $\gamma$ very small has a beneficial effect
on the rate of decay of the eigenvalues of the Gaussian kernel,
as is shown below. An anisotropic,
but stationary generalization of the Gaussian kernel is obtained
by introducing a different positive shape parameter $\gamma_\ell$
for each variable,
\begin{equation} \label{anisoGauss}
K_d(\bx,\bt) = e^{-\gamma_1^2 (x_1-t_1)^2 -\, \cdots \,
- \gamma_d^2 (x_d-t_d)^2}\ \ \ \mbox{for all}
\ \ \ \bx,\bt\in\reals^d.
\end{equation}
\end{subequations}
As evidence of its popularity, we note that this latter kernel is used in
the Gaussian process modeling module of the JMP commercial
statistical software \cite{JMP08}.  In JMP, the values of the~$\gamma_\ell$
are determined in a data-driven way\footnote{In the tractability
literature, the shape parameters $\gamma_\ell$ are called
\emph{product weights}.}.

We stress that the Gaussian kernels are analytic, and the smoothness
parameter $p=\infty$. Therefore one can
hope to obtain convergence rates of the form $\mathcal{O}(n^{-\tau})$
for an arbitrarily large $\tau$. As we shall see, this is indeed the
case. This is shown in Theorem~\ref{thm1} and explained in Section~4.
However, the dependence on $d$ is a function of $\tau$
and only for a relatively small $\tau$ is the dependence on $d$
acceptable.

Given the growing number of applications with moderate to
large dimension, $d$, it is desirable to
have dimension-independent polynomial convergence rates of the form
$Cn^{-p}$ for positive $C$ and $p$,
which corresponds to \emph{strong polynomial tractability},
or at worst, convergence rates that are polynomially
dependent on dimension $d$ and are of the form
$Cd^{\,q}\,n^{-p}$ for positive $C,q$ and $p$,
which corresponds to \emph{polynomial tractability}.

This paper establishes convergence rates with polynomial
or no dimension dependence for the Gaussian kernel introduced
in \eqref{Gausskernelboth}.  The rates are summarized in
Table~\ref{summarytable}.  As explained in Section \ref{funapproxsec}, the absolute error is the $\cl_2$ worst case approximation error based on a Gaussian weight with mean zero and variance $1/2$. The normalized error is the absolute error divided by $\|I_d\|$, were $I_d$ denotes the
embedding between the radial function space $H_d$ and the
$\cl_2$ space. Note that the norm $\|I_d\|$ is the initial error that can
be achieved by the zero algorithm without sampling the functions. The dimension independent convergence rates depend to some extent on which error criterion is used.  They also depend on whether the data available consists
only of function values or, more generally, of arbitrary
linear functionals.  This latter, more generous setting may allow for faster convergence.

\begin{table}
\caption{Error decay rates as a function of sample size $n$ \label{summarytable}}
\centering
\renewcommand{\arraystretch}{1.5}
\begin{tabular}{p{1.3in}||p{2in}|p{2in}}
& \multicolumn{2}{c}{Error Criterion}\\
\raggedleft Data Available & \centering Absolute & \centering Normalized \tabularnewline
\hline
\raggedleft Arbitrary \\ Linear functionals
& \centering $\asymp n^{-\max(r(\bgamma),1/2)}$ \\ Theorem \ref{thm3}
& \centering $\asymp n^{-r(\bgamma)}$ \\ if $r(\bgamma)>0$, Theorem \ref{thm7}
\tabularnewline
\hline
\raggedleft Function values
& \centering $\preceq n^{-\max(r(\bgamma)/[1+1/(2r(\bgamma))],1/4)}$
\\  Theorem \ref{thm4} and \ref{thm5}
& \centering $\preceq n^{-r(\bgamma)/[1+1/(2r(\bgamma))]}$\\
 if $r(\bgamma)>1/2$,   Corollary \ref{cor2}
\tabularnewline
\hline
\end{tabular}
\end{table}

The notation $\preceq n^{-p}$ in
Table  \ref{summarytable} means that for all $\delta > 0$
the error is bounded \emph{above} by $Cn^{-p+\delta}$
for some constant $C$ that is independent of the sample size, $n$,
and the dimension, $d$, but it may depend on $\delta$.
The notation $\succeq n^{-p}$
is defined analogously, and means that the error is
bounded \emph{below} by $Cn^{-p-\delta}$ for all $\delta > 0$.
The notation $\asymp n^{-p}$ means that the error is
both $\preceq n^{-p}$  and $\succeq n^{-p}$.

As can be seen in Table \ref{summarytable}, the convergence rates
depend strongly on how fast the sequence of shape parameters
$\bgamma=\{\gamma_\ell\}_{\ell\in \naturals}$ goes to zero.  The term
$r(\bgamma)$ appearing in Table \ref{summarytable}, is defined by
\begin{equation} \label{wgammaform}
r(\bgamma)=\sup\bigg\{\beta>0\, \bigg |\ \sum_{\ell=1}^{\infty}
\gamma_\ell^{1/\beta} < \infty \bigg\}
\end{equation}
with the convention that the supremum of the empty set is taken to be zero.

For instance, for the isotropic case with
$\gamma_\ell=\gamma>0$
we have $r(\bgamma)=0$, whereas for
$\gamma_\ell=\ell^{-\alpha}$ for a nonnegative $\alpha$ we have
$r(\bgamma)=\alpha$.
If the $\gamma_\ell$ are ordered, that is, $\gamma_{1} \ge \gamma_{2}
\ge \cdots$, then this definition is equivalent to
$$
r(\bgamma)=\sup\big\{\beta\ge0\,|\ \lim_{\ell\to\infty}
\gamma_\ell\,\ell^\beta=0\big\}.
$$

For excellent
\emph{dimension independent convergence} one needs the sequence of
shape parameters to decay to zero quickly, as can be seen in Table
\ref{summarytable}.  These results are derived in Sections
\ref{abserrsec} and \ref{normerrsec}.

\begin{table}
\caption{Number of data, $n(\e,H_d)$, needed to obtain an error tolerance $\varepsilon$ \label{summarycomplexitytable}}
\centering
\renewcommand{\arraystretch}{1.5}
\begin{tabular}{p{1.3in}||p{2in}|p{2in}}
& \multicolumn{2}{c}{Error Criterion}\\
\raggedleft Data Available & \centering Absolute & \centering Normalized \tabularnewline
\hline
\raggedleft Arbitrary \\ Linear functionals
& \centering $\asymp \e^{-\min(1/r(\bgamma),2)}$ \\ Theorem \ref{thm3}
& \centering $\asymp \e^{-1/r(\bgamma)}$ \\ if $r(\bgamma)>0$, Theorem \ref{thm7}
\tabularnewline
\hline
\raggedleft Function values
& \centering $\preceq \e^{-\min(1/r(\bgamma)+1/[2r^2(\bgamma)],4)}$
\\  Theorem \ref{thm4} and \ref{thm5}
& \centering $\preceq \e^{-1/r(\bgamma)-1/[2r^2(\bgamma)]}$\\
 if $r(\bgamma)>1/2$,   Corollary \ref{cor2}
\tabularnewline
\hline
\end{tabular}
\end{table}

While writing the error as a function of the sample size is common in the numerical analysis literature, the computational complexity literature looks at the number of data required to obtain a given error tolerance.  Let
$n(\e,H_d)$ denote the minimal number of function values or linear
functionals that are needed to compute an $\e\cdot{\rm CRI}_d$ approximation.
Here, ${\rm CRI}_d=1$ for the absolute error criterion, and $
{\rm CRI}_d=\|I_d\|$ for the normalized error criterion.  Again, $\|I_d\|$ is the initial error that can
be achieved by the zero algorithm without sampling the functions.  The tractability results presented in this paper are summarized in Table \ref{summarycomplexitytable}.

For the absolute error and algorithms that use arbitrary
linear functionals, we prove \emph{strong polynomial tractability}
for \emph{all} choices of shape parameters $\gamma_\ell$.
Furthermore, the exponent $2$ of $\e^{-1}$ is best possible for all $\gamma_\ell$'s that go to zero no
faster that $\ell^{-2}$.  For the absolute error and algorithms that use function values, we still have strong polynomial tractability with exponent at most $4$.

For the normalized error, the situation is much worse. If the sequence of shape parameters tends to zero fast enough, we still have strong polynomial tractability.  However, for the isotropic case it follows
that $n(\e,H_d)$ does \emph{not} depend polynomially
on $\e^{-1}$ and $d$.
For algorithms using arbitrary linear functionals,
we have \emph{quasi-polynomial
tractability}, i.e., there are positive $C$ and $t$ such that
$$
n(\e,H_d)\le C\,\exp(t\,\left(1+\ln\,d)\,(1+\ln\,\e^{-1})\right)
\ \ \ \mbox{for all}\ \  \ \e\in(0,1)\ \ \mbox{and}\ \ d\in\naturals.
$$
Furthermore, the smallest $t$ is roughly\footnote{In this paper, by
  $\ln$ we mean the natural logarithm of base $e$.}
$$
t=2/{\ln\left(
\frac{1+2\gamma^2+\sqrt{1+4\gamma^2}}{2\gamma^2} \right) }.
$$

As a prelude to deriving these convergence
and tractability results, the next section reviews
some principles of function approximation on Hilbert spaces.
Section \ref{radfunsec} applies these principles to the Gaussian kernel.

\section{Function Approximation} \label{funapproxsec}
Let $H_d=H(K_d)$ denote a reproducing kernel Hilbert space of
real functions defined on a Lebesgue measurable
set $D_d\subseteq \reals^d$.  The goal is to accurately
approximate any function in $H_d$ given a finite number of
data about it.  The reproducing kernel
$$
K_d:D_d\times D_d\to\reals
$$
is symmetric, positive definite and reproduces function values.
This means that
for all $n\in \naturals$, $\bx,\bt,\bx_1,\bx_2,\dots,\bx_n\in D_d$,
$\c=(c_1,c_2,\dots,c_n)\in\reals^n$ and $f\in H_d$,
the following properties hold:
\begin{subequations} \label{Kassump}
\begin{eqnarray}
K_d(\cdot,\bx)&\in&H_d,\\
K_d(\bx,\bt)&=&K_d(\bt,\bx),\\
\sum_{i=1}^n K_d(\bx_i,\bx_j) c_i c_j &\ge& 0,\\
f(\bx) &=& \il f, K_d(\cdot,\bx)\ir_{H_d}.
\end{eqnarray}
\end{subequations}
For an arbitrary $\bx\in D_d$
consider the linear functional $L_\bx(f)=f(\bx)$ for all $f\in H_d$.
Then $L_\bx$  is
continuous and $\|L_\bx\|_{H_d^*}=K_d^{1/2}(\bx,\bx)$.
The reader may find these and other properties in
e.g.,~\cite{BTA04,Wah90}.
Many reproducing kernels are used in practice.
A popular choice is the Gaussian kernel defined in
\eqref{Gausskernelboth} for which $D_d=\reals^d$.

It is assumed that $H_d$ is continuously embedded in the space
$\cl_2=\cl_2(D_d,\rho_d)$ of square Lebesgue integrable functions.
Here, $\rho_d$
is a probability density function, i.e., $\rho_d\ge 0$ and
$\int_{D_d}\rho_d(\bt)\,{\rm d}\bt=1$. The norm in the space $\cl_2$ is
given by
$$
\|f\|_{\cl_2}=\left(\int_{D_d}f^2(\bt)\,\rho_d(\bt)\,{\rm d}\bt\right)^{1/2}.
$$
Continuous embedding means that the linear embedding operator
$I_d:H_d\to \cl_2$ given by $I_df=f$ is continuous,
$$
\|I_df\|_{\cl_2}\le \|I_d\|\ \|f\|_{H_d}\ \ \ \ \mbox{for all}\ \ \  f\in H_d.
$$
Observe that
\begin{eqnarray*}
\|I_df\|^2_{\cl_2}&=&\int_{D_d}f^2(\bt)\,\rho_d(\bt)\,{\rm d}\bt=
\int_{D_d}\il f,K_d(\cdot,\bt)\ir^2_{H_d}\,\rho_d(\bt)\,{\rm d}\bt\\
&\le&\|f\|^2_{H_d}\,\int_{D_d}K_d(\bt,\bt)\,\rho_d(\bt)\,{\rm d}\bt.
\end{eqnarray*}
Hence, it is enough to assume that
\begin{equation}\label{embedding}
\int_{D_d}K_d(\bt,\bt)\,\rho_d(\bt)\,{\rm d}\bt<\infty
\end{equation}
to guarantee that $I_d$ is continuous, and obviously
$$
\|I_d\|\le \left[\int_{D_d}K_d(\bt,\bt)\,\rho(\bt)\,{\rm
    d}\bt\right]^{1/2}.
$$

Functions in $H_d$ are approximated by linear
algorithms\footnote{It is well
known that adaption and nonlinear algorithms do not help
for approximation of linear problems. A linear problem
is defined as a linear operator and we approximate its values
over a set that
is convex and balanced. The typical
example of such a set is the unit ball as taken in this paper.
Then among all algorithms that
use linear adaptive functionals, the worst case error is minimized
by a linear algorithm that uses nonadaptive linear functionals.
Adaptive choice of a linear functional means that the choice of $L_j$
in~\eqref{linearnonadaptive} may depend on the already computed values
$L_i(f)$ for $i=1,2,\dots,j-1$. That is why in our case,
the restriction to linear algorithms of the
form~\eqref{linearnonadaptive} can be done without loss of
generality, for more detail see, e.g., \cite{TWW88}.}
\begin{equation}\label{linearnonadaptive}
A_n(f)=\sum_{j=1}^nL_j(f)a_j\ \ \ \mbox{for all}\ \ \  f\in H_d
\end{equation}
for some continuous linear functionals $L_j\in H_d^*$,
and functions $a_j\in \cl_2$.
The \emph{worst case error} of the algorithm $A_n$ is then defined as
$$
e^{\rm wor}(A_n)=\sup_{\|f\|_{H_d}\le1}\|f-A_n(f)\|_{\cl_2}.
$$

The linear algorithms $A_n$ considered here are based on function
data $L_j(f)$, where the continuous linear functionals $L_j$ may
belong to one of two classes. The first class, denoted $\lstd$, is comprised
only of function values and is called \emph{standard}.
That is, $L_j\in\lstd$ iff $L_j(f)=f(\bt_j)$ for all $f\in
H_d$ for some $\bt_j\in D_d$. The second class, denoted $\lall$,
is comprised of arbitrary continuous functionals and is
called \emph{linear}. That is, $L_j\in\lall$ iff $L_j\in H_d^*$.
Obviously, $\lstd\subseteq\lall$.

The aim is to determine how small the worst case error can be by
choosing linear algorithms with only $n$ linear functionals
either from $\lstd$ or~$\lall$.  The \emph{$n$th minimal worst case
  error} is defined as
\begin{equation*}
e^{\rm wor-\stdall}(n,H_d)
=\inf_{A_n \ {\rm with}\ L_j\in\Lambda^{\stdall} }e^{\rm
  wor}(A_n), \quad \stdall \in \{{\rm std},{\rm all}\}.
\end{equation*}
Here and below, for notational simplicity
$\stdall$ denotes either the standard or linear setting.
Clearly, $e^{\rm wor-all}(n,H_d)\le e^{\rm wor-std}(n,H_d)$
since the former uses a larger class of function data.

The case $n=0$ means that no linear functionals of $f$ are used to
construct the algorithm.  It is easy to see that
the best  algorithm possible is $A_0=0$, and then
$$
e^{\rm wor-\stdall}(0,H_d)=\|I_d\|.
$$
The minimal error for $n=0$ is called the initial error and it only
depends on the formulation of the problem.

This article addresses two problems: convergence and tractability.
The former considers how fast the error vanishes as $n$ increases,
and the latter considers how the error depends on the dimension, $d$,
as well as the number of data, $n$.

\vskip 1pc
\noindent {\bf Problem 1:\ Rate of Convergence}\
\vskip 1pc
We would like to know how fast
$e^{\rm wor-\stdall}(n,H_d)$
goes to zero
as $n$ goes to infinity. In particular, we study the rate of
convergence (defined by~\eqref{wgammaform}) of
the sequence $\{e^{\rm wor-\stdall}(n,H_d)\}_{n\in\naturals}$. Since the numbers $e^{\rm wor-\stdall}(n,H_d)$
are ordered, we have
\begin{equation}\label{rate}
r^{\rm wor-\stdall}(H_d):=r\left(\{e^{\rm wor-\stdall}(n,H_d)\}\right)
=\sup\left\{\beta\ge0\,|\ \ \lim_{n\to\infty}
e^{\rm wor-\stdall}(n,H_d)
\,n^\beta=0\right\}.
\end{equation}
Roughly speaking, the rate of convergence is the largest $\beta$
for which the $n$th minimal errors behave no worse than $n^{-\beta}$.
For example, if $e^{\rm wor-\stdall}(n,H_d)=n^{-\alpha}$ for a
positive $\alpha$  then $r^{\rm wor-\stdall}(H_d)=\alpha$.
Under this definition, even sequences of the form
$e^{\rm wor-\stdall}(n,H_d)=n^{-\alpha} \ln^p n$ for an arbitrary $p$
still have $r^{\rm wor-\stdall}(H_d)=\alpha$.
On the other hand, if $e^{\rm wor-\stdall}(n,H_d)=q^{n}$ for a
number $q\in(0,1)$ then $r^{\rm wor-\stdall}(H_d)=\infty$.

Obviously, $r^{\rm wor-all}(H_d)\ge r^{\rm wor-std}(H_d)$. We would
like to know both rates and whether
$$
r^{\rm wor-all}(H_d)>r^{\rm wor-std}(H_d),
$$
i.e., whether  $\lall$ admits a better rate of convergence
than $\lstd$.

\vskip 1pc
\noindent {\bf Problem 2:\  Tractability}\
\vskip 1pc
Assume that there is a sequence of spaces $\{H_d\}_{d \in \naturals}$ and embedding
operators $\{I_d\}_{d \in \naturals}$.
In this case, we would like to know how the minimal errors
$e^{\rm wor-\stdall}(n,H_d)$
depend not only on $n$ but also on $d$.

More precisely, we consider the~\emph{absolute} and~\emph{normalized}
error criteria.
For a given (small) positive $\e\in(0,1)$ we want to find
an algorithm $A_n$ with the smallest~$n$ for which the error does not
exceed $\e$ for the absolute error criterion, and does not exceed
$\e\,\|I_d\|$ for the normalized error criterion. That is,
$$
n^{\rm wor-\absnorm-\stdall}(\e,H_d)
= \min\left\{n\,|\
e^{\rm wor-\stdall}(n,H_d)
\le\e\,{\rm CRI}^{\absnorm}_d\right\}, \quad \absnorm \in \{ {\rm abs},{\rm norm} \},
$$
where ${\rm CRI}^{\rm abs}_d=1$ for the absolute error criterion
and ${\rm CRI}^{\rm nor}_d=\|I_d\|$ for the normalized error criterion.

Let $\mathcal{I}=\{I_d\}_{d \in \naturals}$ denote the sequence of function approximation
problems. We say that $\mathcal{I}$ is~\emph{polynomially tractable} iff
there exist numbers $C$, $p$ and $q$ such that
$$
n^{\rm wor-\absnorm-\stdall}(\e,H_d)\le C\,d^{\,q}\,\e^{-p}\ \ \
\mbox{for all}\ \ \ d\in \naturals\ \ \mbox{and}\ \ \ \e\in(0,1).
$$
If $q=0$ above then we say that $\mathcal{I}$ is~\emph{strongly polynomially
  tractable} and the infimum of~$p$ satisfying the bound above is
called the~\emph{exponent} of strong polynomial tractability.

The essence of polynomial tractability is to guarantee that
a polynomial number of linear functionals is enough to satisfy
the function approximation problem to within $\e$.
Obviously, polynomial tractability depends on which class, $\lall$ or
$\lstd$, is considered and whether the absolute or normalized error is used.
As shall be shown, the results on polynomial tractability depend on
the cases considered.

The property of strong polynomial tractability is especially
challenging since then the number of linear functionals needed for an
$\e$-approximation is independent of~$d$. The reader may suspect that
this property is too strong and cannot happen for function
approximation. Nevertheless, there are positive results to report on
strong polynomial tractability.

Besides polynomial tractability, there are the
somewhat less demanding concepts such as
quasi-polynomial tractability and weak tractability.
The problem $\mathcal{I}$ is~\emph{quasi-polynomially tractable} iff
there exist numbers $C$ and $t$ for which
$$
n^{\rm wor-\absnorm-\stdall}(\e,H_d)
\le C\,\exp\left(t\,
\ln(1+d)\,\ln(1+\e^{-1})\right)
$$
for all $d\in \naturals$ and $\e>0$.
The exponent of quasi-polynomial tractability is defined as
the infimum of $t$ satisfying the bound above.
Finally, $\mathcal{I}$ is~\emph{weakly tractable} iff
$$
\lim_{\e^{-1}+d\to\infty}\frac{\ln\,n^{\rm
    wor-\absnorm-\stdall}(\e,H_d)}{\e^{-1}+d}=0.
$$
Note that for a fixed $d$, quasi-polynomial tractability means that
$$
n^{\rm wor-\absnorm-\stdall}(\e,H_d)=\mathcal{O}\left(\e^{-t(1+\ln\,d)}\right)
\ \ \ \mbox{as}\ \ \ \e\to0.
$$
Hence, the exponent of $\e^{-1}$ may now weakly depend on $d$ through
$\ln\,d$.
On the other hand, weak tractability only means that we do not have
exponential dependence on $\e^{-1}$ and $d$.

We will report about quasi-polynomial and weak tractability in the
case when polynomial tractability does not hold. As before,
quasi-polynomial and weak tractability depend on which class $\lall$ or
$\lstd$ is considered and on the error criterion.

Motivation of tractability study and more on tractability concepts can be
found in~\cite{NW08}. Quasi-polynomial tractability has been recently
studied in~\cite{GW10}.

\vskip 2pc
We end this section by briefly reviewing some general results related to the problems of convergence and tractability mentioned above.
For a given design, i.e., given continuous linear functionals $L_1, \ldots, L_n$, it is known how to find
functions $a_1, \ldots, a_n$ for which the worst case error of $A_n$ is minimized.
The optimal algorithm, $S_n$, should be taken as the~\emph{spline} or
the~\emph{minimal norm interpolant}, see e.g.~Section 5.7 of~\cite{TWW88}.  The spline algorithm was briefly mentioned in the introduction.  It is described in more generality here.

For given $y_j=L_j(f)$ for $j=1,2,\dots,n$, we take
$S_n(f)$ as an element of $H_d$ that satisfies the conditions
\begin{eqnarray*}
L_j(S_n(f))&=&y_j\ \ \ \ \ \ \ \qquad \qquad\qquad \qquad \qquad
\mbox{for }\ \ \ j=1,2,\dots,n,\\
\|S_n(f)\|_{H_d}&=&\inf_{g\in H_d,\
L_j(g)=y_j,\  j=1,2,\dots,n}\|g\|_{H_d}.
\end{eqnarray*}
The construction of $S_n(f)$ may be done by solving a linear equation
$\mK \bc=\by$, where $\by=(y_1,y_2,\dots,y_n)^T$ and the $n\times n$
matrix is given by
$$
\mK=(k_{i,j})_{i,j=1}^n
\ \ \ \mbox{with}\ \ \
k_{i,j}=L_i(g_j)\ \ \mbox{and}\ \  \ g_j(\bx)=L_jK_d(\cdot,\bx).
$$
Then
$$
S_n(f)(\bx)=
\bk^T(\bx)\mK^{-1}\by \ \ \ \mbox{with}\ \ \
\bk(\bx)=(L_iK_d(\cdot,\bx))_{i=1}^n
$$
and
$$
e^{\rm wor}(S_n)=\sup_{\|f\|_{H_d}\le1,\
  L_j(f)=0,\,j=1,2,\dots,n}\|f\|_{\cl_2}.
$$
Note that depending on the choice of linear functionals
$L_1,\ldots,L_n$ the matrix $\mK$ may not necessarily be invertible,
however, the solution $\bc=\mK^{-1}\by$ is always well defined as
the vector of minimal Euclidean norm which satisfies
$\mK\bc=\by$.

The spline enjoys more optimality properties. For
instance, it minimizes the~\emph{local} worst case error.  Roughly speaking this means that for each $\bx\in D_d$, the worst possible pointwise error $|f(\bx) - A_n(f)(\bx)|$ over the unit ball of functions $f$ is minimized over all possible $A_n$ by choosing $A_n=S_n$.  We do not elaborate more on this point.

It is non-trivial to find the linear functionals $L_j$ from the 
class $\lstd$ that  minimize the error of 
the spline algorithm $S_n$. For the class $\lall$, 
the optimal design is known, at least theoretically, 
see again e.g.,~\cite{TWW88}.
Namely, let $W_d=I_d^*I_d:H_d\to H_d$,
where $I_d^*: \cl_2 \to H_d$ denotes the adjoint of 
%H                ^^
the imbedding operator, i.e., the operator satisfying 
$\il f, I^*_d h \ir_{H_d}=\il I_d f, h \ir_{\cl_2}$ 
for all $f \in H_d$ and $h \in \cl_2$. As a consequence, $W_d$ is a
self adjoint and positive definite linear operator given by
$$
W_d(f)=\int_{D_d}f(\bt)\,K_d(\cdot,\bt)\,\rho_d(\bt)\,{\rm d}\bt\ \ \
\mbox{for all}\ \ \ f\in H_d.
$$
Clearly,
$$
\il f,g\ir_{\cl_2}=
\il I_df,I_dg\ir_{\cl_2}=\il W_df,g\ir_{H_d}=\il f,W_dg\ir_{H_d}
\ \ \ \mbox{for all}\ \ \ f,g\in
H_d.
$$
It is known that $\lim_{n\to\infty}
e^{\rm wor-all}(n,H_d)=0$
iff $W_d$ is compact.
In particular,~\eqref{embedding} implies that $W_d$ is compact.

Assuming that $W_d$ is compact, let us define its eigenpairs by
$(\lambda_{d,j},\eta_{d,j})$, where
the eigenvalues are ordered, $\lambda_{d,1}\ge\lambda_{d,2}\ge\cdots$, and
$$
W_d\,\eta_{d,j}=\lambda_{d,j}\,\eta_{d,j} \ \ \ \mbox{with}\ \ \
\il \eta_{d,j},\eta_{d,i}\ir_{H_d}=\delta_{i,j} \ \ \mbox{for all}\ \
i,j\in\naturals.
$$
Note also that for any $f\in H_d$ we have
$$
\il f,\eta_{d,j}\ir_{\cl_2}=\il I_df,I_d\eta_{d,j}\ir_{\cl_2}=
\il f,W_d\eta_{d,j}\ir_{H_d}=\lambda_{d,j}\il f,\eta_{d,j}\ir_{H_d}.
$$
Taking $f=\eta_{d,i}$ we see that $\{\eta_{d,j}\}$ 
us a set of orthogonal functions in $\cl_2$.
For simplicity and without loss of generality we assume that all $\lambda_{d,j}$
are positive\footnote{Otherwise, we should switch to a subspace of
$H_d$ spanned by eigenfunctions corresponding to $k$ positive
eigenvalues, and replace $\naturals$ by $\{1,2,\dots,k\}$.}.
Letting
$$
\varphi_{d,j}=\lambda_{d,j}^{-1/2}\eta_{d,j}\ \ \
\mbox{for all}\ \ \ j\in\naturals
$$
we obtain an orthonormal sequence $\{\varphi_{d,j}\}$ in $\cl_2$.
Since $\{\eta_{d,j}\}$ is a complete
orthonormal basis of $H_d$ we have
\begin{equation}\label{formofkd}
K_d(\bx,\bt)=\sum_{j=1}^\infty\eta_{d,j}(\bx)\,\eta_{d,j}(\bt)=
\sum_{j=1}^\infty\lambda_{d,j}\,\varphi_{d,j}(\bx)\,
\varphi_{d,j}(\bt) \ \ \
\mbox{for all}\ \ \ \bx,\bt\in D_d.
\end{equation}
If~\eqref{embedding} holds then
\begin{equation}\label{sumeig}
\sum_{j=1}^\infty\lambda_{d,j}=\int_{D_d}K_d(\bt,\bt)\,\rho_d(\bt)\,
{\rm d}\bt <\infty.
\end{equation}
This means that \eqref{embedding} implies that
$W_d$ is also a finite trace operator.

It is
known that the best choice of $L_j$ for the class $\lall$
is $L_j=\il \cdot,\eta_{d,j}\ir_{H_d}$. Then the spline algorithm $S_n$ with
the minimal worst case error is defined using the eigenfunctions corresponding to the $n$ largest eigenvalues:
$$
S_n(f)=\sum_{j=1}^n\il f,\eta_{d,j}\ir_{H_d}\eta_{d,j}\ \ \ \ \
\mbox{for all}\ \ \  f\in H_d,
$$
and
$$
e^{\rm wor}(S_n)=e^{\rm wor-all}(n,H_d)=\sqrt{\lambda_{d,n+1}}\ \
\ \ \ \mbox{for all}\ \ \ n\in\naturals.
$$
The last formula for $n=0$ yields that the initial error
is $\|I_d\|=\sqrt{\lambda_{d,1}}$.

The results for the class $\lall$ are useful for finding rates of
convergence as well as necessary and
sufficient conditions on
polynomial, quasi-polynomial and weak tractability in terms of the
behavior of the eigenvalues $\lambda_{d,j}$. This has already been
done in a number
of papers or books, and we will report these results later for spaces
studied in this paper.
For the class $\lstd$, the situation is much harder although there are
papers that relate rates of convergence and tractability conditions
between classes $\lall$ and $\lstd$. Again we report these results later.

\section{Radial Function Spaces} \label{radfunsec}
The focus of this article is on reproducing kernels
that are \emph{translation invariant} or \emph{stationary}, i.e.,
$$
K_d(\bx,\bt) = \tK_d(\bx-\bt)\ \ \ \mbox{for all}
\ \ \  \bx,\bt\in D_d=\reals^d.
$$
An even more special case is for
\emph{radially symmetric} or \emph{isotropic} kernels,
i.e.,
$$
K_d(\bx,\bt) = \kappa(\|\bx-\bt\|^2_2)\ \ \ \mbox{with}\ \ \
\|\bx-\bt\|_2^2=\sum_{\ell=1}^d(x_\ell-t_\ell)^2.
$$
Here, $\tK_d$ and $\kappa$ are chosen such that $K_d$ is a reproducing
kernel.

Isotropic kernels also go by the name \emph{radial basis functions},
and the spaces $H(K_d)$ are referred to as~\emph{radial function spaces}.
Stationary or isotropic kernels are common in the literature on
computational mathematics \cite{Buh03,Fas07,Wen05}, statistics \cite{BTA04,Ste99,Wah90}, statistical learning \cite{RW06,SC08}, and engineering applications \cite{FSK08}.

A popular isotropic kernel is the Gaussian kernel, defined in \eqref{Gausskernelboth}, which has both an isotropic version,
\begin{equation*}
K_d(\bx,\bt) = e^{-\gamma^2 \|\bx-\bt\|^2} \ \ \ \mbox{for all}
\ \ \ \bx,\bt\in\reals^d ,
\end{equation*}
and a more general anisotropic version,
\begin{equation*}
K_d(\bx,\bt) = e^{-\gamma_1^2 (x_1-t_1)^2 -\, \cdots \,
- \gamma_d^2 (x_d-t_d)^2}\ \ \ \mbox{for all}
\ \ \ \bx,\bt\in\reals^d.
\end{equation*}
As alluded to in the introduction, the shape parameter, 
$\gamma$ or $\bgamma=\{\gamma_\ell\}_{\ell \in \naturals}$, 
which functions as an inverse length scale, plays an important 
role in the tractability of function approximation.  
Choosing the $\gamma_\ell$ to decay quickly has a beneficial effect
on the rate of decay of the eigenvalues of the Gaussian kernel, as
we shall see below.
On the other hand, a small value of $\gamma$
leads to a huge condition number of the matrix $\mK$
and may result in severe numerical instabilities. While this is an issue
that is very important for practical implementations, and has received some attention,
we will not discuss it any further here.

We now analyze the function approximation problem
for the Hilbert space $H_d=H(K_d)$ with the isotropic Gaussian kernel
$K_d$ given by~\eqref{Gausskernel} or, more generally,
with the anisotropic Gaussian
kernel given by~\eqref{anisoGauss}. For the space
$\cl_2(\reals^d,\rho_d)$ we take the Gaussian weight with zero mean and
variance $1/2$, i.e.,
$$
\rho_d(\bt)=\frac1{\pi^{d/2}}\exp\left(-(t_1^2+t_2^2+\cdots+t_d^2\right))\ \ \
\ \ \mbox{for all} \ \ \ \bt\in \reals^d.
$$
Note that $K_d(\bt,\bt)=1$ for all $\bt\in\reals^d$, and
therefore
$$
\int_{\reals^d}K_d(\bt,\bt)\,\rho_d(\bt)\,{\rm d}\bt=1,
$$
so that~\eqref{embedding} holds. This means that the embedding $I_d$
is continuous, the operator $W_d$ is compact and a finite trace
operator with
\begin{equation}\label{sumeigone}
\sum_{j=1}^\infty\lambda_{d,j}=1,
\end{equation}
by \eqref{sumeig}.

Observe that~\eqref{embedding} holds for all translation invariant
kernels since
$$
\int_{\reals^d}K_d(\bt,\bt)\,\rho_d(\bt)\,{\rm d}\bt=\tK_d(\bzero),
$$
but it can now depend on $d$. For radially symmetric kernels we have
$$
\int_{\reals^d}K_d(\bt,\bt)\,\rho_d(\bt)\,{\rm d}\bt=\kappa(0),
$$
and it is independent of $d$.

Since a Gaussian kernel $K_d$ is of the product form,
the space $H_d$ is the tensor
product of the Hilbert spaces of univariate spaces with the kernels
$e^{-\gamma_\ell^2(x-t)^2}$ for $x,t\in\reals$.
This also implies that the operator $W_d$ is of the product form and
its eigenpairs are products of the corresponding eigenpairs for the
univariate cases.

Consider now $d=1$, and the space $H(K_1)$ with
$K_1(x,t)=e^{-\gamma^2(x-t)^2}$.
Then the eigenpairs $\left(\tlambda_{\gamma,j},\eta_{\gamma,j}\right)$
of $W_1$ are known, see \cite{RW06}.
Note that we have introduced the notation 
$\tlambda_{\gamma,j}$ to emphasize the dependence 
of the eigenvalues on $\gamma$ in the following 
discussion (while the dependence on $d$ 
has temporarily dropped from the notation).
We have
\begin{equation*}
\tlambda_{\gamma,j} =
\frac{1}{\sqrt{\frac 12 (1 + \sqrt{1 + 4 \gamma^2}) + \gamma^2}}
\left(\frac{\gamma^2}{\frac 12 (1 + \sqrt{1 + 4 \gamma^2}) +
\gamma^2} \right)^{j-1}= (1-\omega_{\gamma} )\,\omega^{j-1}_{\gamma},
\end{equation*}
where
\begin{equation}
\label{omegadef}
\omega_\gamma=\frac{\gamma^2}{\tfrac12(1+\sqrt{1+4\gamma^2})+\gamma^2},
\end{equation}
and $\eta_{\gamma,j}=\sqrt{\tlambda_{\gamma,j}}\,\varphi_{\gamma,j}$ with
$$
\varphi_{\gamma,j}(x) =
\sqrt{\frac{(1+4\gamma^2)^{1/4}}{2^{j-1}(j-1)!}}
\exp\left(- \frac{\gamma^2 x^2}
{\frac 12 (1 + \sqrt{1 + 4 \gamma^2})} \right)
H_{j-1}\left((1+4\gamma^2)^{1/4} x\right),
$$
where $H_{j-1}$ is the Hermite polynomial of degree $j-1$, given by
$$
H_{j-1}(x) = (-1)^{j-1} e^{x^2}\frac{{\rm d}^{j-1}}{{\rm  d}x^{j-1}}e^{-x^2}
\ \ \ \mbox{for all}\ \ \ x\in\reals,
$$
so that
$$
\int_{\reals} H_{j-1}^2(x)\, e^{-x^2}\, {\rm d}x = \sqrt{\pi}\,
2^{j-1} (j-1)! \qquad \mbox{for}\ \ j=1,2,\ldots\, .
$$
Obviously, we have
$$
\il \eta_{\gamma,i},\eta_{\gamma,j}\ir_{H(K_1)}=
\il \varphi_{\gamma,i},\varphi_{\gamma,j}\ir_{\cl_2} = \delta_{ij},
$$
and applying~\eqref{formofkd} we obtain
$$
K_1(x,t) = e^{-\gamma^2 (x-t)^2} = \sum_{j=1}^{\infty}
\tlambda_{\gamma,j}
\varphi_{\gamma,j} (x)\varphi_{\gamma,j}(y)\ \ \
\ \ \mbox{for all}\ \ \ x,t\in\reals.
$$

Note that the eigenvalues $\tlambda_{\gamma,j}$ are ordered.
The largest eigenvalue is
$$
\tlambda_{\gamma,1}=
1 - \omega_{\gamma} =\sqrt{\frac2{1+\sqrt{1+4\gamma^2}+2\gamma^2}}=
1-\gamma^2+\mathcal{O}(\gamma^4)\ \ \
\mbox{as}\ \ \ \gamma\to0.
$$
Furthermore,
\begin{equation}\label{asymeig}
\tlambda_{\gamma,j}=\left(1-\gamma^2+\mathcal{O}(\gamma^4)\right)
\left(\frac{\gamma^2}{1-\gamma^2+\mathcal{O}(\gamma^4)}\right)^{j-1}
\ \ \ \mbox{for}\ \ j=1,2,\dots\,.
\end{equation}
The space $H(K_1)$ consists of analytic functions for which
$$
\|f\|^2_{H(K_1)}=\sum_{j=1}^\infty\il f,\eta_{\gamma,j}\ir^2_{H(K_1)}
= \sum_{j=1}^\infty\frac 1{\tlambda_{\gamma,j}}
\il f,\varphi_{\gamma,j}\ir^2_{\cl_2} <\infty.
$$
This means that the coefficients of $f$ in the space $\cl_2$ decay
exponentially fast. The inner product is obviously given as
\begin{multline*}
\il f,g\ir_{H(K_1)}=
\sum_{j=1}^\infty \frac1{\tlambda_{\gamma,j}}
\int_{\reals}f(t)\,\frac{\varphi_{\gamma,j}(t)}{\sqrt{\pi}}\,e^{-t^2}\,{\rm d}t
\int_{\reals}g(t)\,\frac{\varphi_{\gamma,j}(t)}{\sqrt{\pi}}\,e^{-t^2}\,{\rm d}t
\\ \mbox{for all}\ \ \ f,g\in H(K_1).
\end{multline*}
The reader may find more about the characterization of the space
$H(K_1)$ in~\cite{SHS06}.

\vskip 1pc
For $d>1$, let $\bgamma=\{\gamma_\ell\}_{\ell \in \naturals}$
and $\bj=(j_1,j_2,\dots,j_d)\in\naturals^d$.
As already mentioned, the eigenpairs
$\left(\tlambda_{d,\bgamma,\bj},\eta_{d,\bgamma,\bj}\right)$ of $W_d$ are
given by the products
\begin{align}
\label{tlambarbd}
\tlambda_{d,\bgamma,\bj}& =\prod_{\ell=1}^d\tlambda_{\gamma_{\ell},j_\ell}
=\prod_{\ell=1}^d
\frac{1}{\sqrt{\frac 12 (1 + \sqrt{1 + 4 \gamma_\ell^2}) + \gamma_\ell^2}}
\left(\frac{\gamma_\ell^2}{\frac 12 (1 + \sqrt{1 + 4 \gamma_\ell^2}) +
\gamma_\ell^2} \right)^{j_\ell-1}\\
\nonumber
& = \prod_{\ell=1}^d (1-\omega_{\gamma_{\ell}} )\,\omega^{\,j_\ell-1}_{\gamma_{\ell}},
\end{align}
where $\omega_{\gamma}$ is defined above in \eqref{omegadef}, and
$$
\eta_{d,\bgamma,\bj}=\prod_{\ell=1}^d\sqrt{\tlambda_{\gamma_\ell,j_\ell}}\,
\varphi_{\gamma_\ell,j_\ell}
$$
with
$$
\il \eta_{d,\bgamma,\bi},\eta_{d,\bgamma,\bj}\ir_{H_d}=
\il \varphi_{\bgamma,\bi},\varphi_{\bgamma,\bj}\ir_{\cl_2} = \delta_{\bi \bj}.
$$

This section ends with a lemma describing the convergence of
the sums of powers of the eigenvalues for the multivariate problem,
and how these sums depend on the dimension, $d$.
This lemma is used in several of the theorems on convergence
and tractability in the following sections.

In the next sections, it
will be convenient to reorder the sequence of eigenvalues
$\{\tlambda_{d,\bgamma,\bj}\}_{\bj\in\naturals^d}$
as the sequence $\{\lambda_{d,j}\}_{j\in\naturals}$ with
$\lambda_{d,1}\ge\lambda_{d,2}\ge \cdots$.
Obviously, for the univariate case, $d=1$, we have
$\lambda_{1,j}=\tlambda_{1,\gamma_1,j}$ for all $j\in\naturals$,
but for the multivariate case, $d>1$, the correspondence between
$\lambda_{d,j}$ and $\tlambda_{d,\bgamma,\bj}$ is more complex.
Obviously,
$$
\lambda_{d,1}=\prod_{\ell=1}^d\left(1-\omega_{\gamma_\ell}\right).
$$
We now present a simple estimate of $\lambda_{d,n+1}$ that will be
needed for our analysis.

\begin{lemma} \label{lem1}  Let $\tau>0$.
Consider the Gaussian kernel with the
sequence of shape parameters
$\bgamma = \{\gamma_\ell\}_{\ell \in \naturals}$.
The sum of the $\tau^{\text{th}}$ power of the eigenvalues for
the $d$-variate case, $d\ge1$, is
\begin{equation} \label{sumpoweig}
\sum_{j=1}^\infty
\lambda_{d,j}^\tau = \sum_{\bj\in \naturals^d}\tlambda_{d,\bgamma,\bj}^\tau
=\prod_{\ell=1}^d\,\left(\sum_{j=1}^\infty
\tlambda_{\gamma_\ell,j}^\tau\right) = \prod_{\ell=1}^d
\frac{(1-\omega_{\gamma_{\ell}})^{\tau}}{1-\omega_{\gamma_\ell}^\tau}
\begin{cases} > 1, & 0 < \tau < 1, \\
=1, & \tau = 1.
\end{cases}
\end{equation}
The $(n+1)^{\text{st}}$ largest eigenvalue  satisfies
\begin{equation} \label{nplusfirst}
\lambda_{d,n+1}\le \frac1{(n+1)^{1/\tau}}\,\prod_{\ell=1}^d
\frac{1-\omega_{\gamma_{\ell}}}{(1-\omega_{\gamma_\ell}^\tau)^{1/\tau}}.
\end{equation}
\end{lemma}
\begin{proof} Equation \eqref{sumpoweig} follows directly from the
formula for $\tlambda_{d,\bgamma,\bj}$ in  \eqref{tlambarbd}.
{}From the definition of $\omega_{\gamma}$ in \eqref{omegadef}
it follows that $0 < \omega_{\gamma} < 1$ for all $\gamma >0$.
For $\tau\in(0,1)$, consider the function
$$
f(\omega)=(1-\omega)^\tau-1 +\omega^\tau\ \ \ \mbox{for all}
\ \ \ \omega\in[0,1].
$$
Clearly, $f$ is concave and vanishes at $0$ and $1$,
and therefore $f(\omega)>0$ for all $\omega\in(0,1)$.
This yields the lower bound on the sum
of the power of the univariate eigenvalues.

The ordering of the eigenvalues $\lambda_{d,j}$ implies that
\[
\lambda_{d,n+1} \le
\bigg(\frac{1}{n+1} \sum_{j=1}^{n+1} \lambda_{d,j}^\tau \bigg)^{1/\tau}
\le \bigg(\frac{1}{n+1} \sum_{j=1}^{\infty} \lambda_{d,j}^\tau
\bigg)^{1/\tau} =\frac{1}{(n+1)^{1/\tau}}
\bigg(\sum_{j=1}^{\infty} \lambda_{d,j}^\tau \bigg)^{1/\tau}.
\]
This yields the upper bound on the $n+1^{\text{st}}$ largest
eigenvalue in \eqref{nplusfirst}, and completes the proof.
\end{proof}

The main point of \eqref{nplusfirst} is that this estimate holds
for all positive
$\tau$. This means that $\lambda_{d,n+1}$ goes to zero faster than any
polynomial in $(n+1)^{-1}$.

\section{Rates of Convergence for Gaussian Kernels} \label{convergesec}

In this section we consider the function approximation problem
for the Hilbert space $H_d=H(K_d)$ with the anisotropic Gaussian kernel
given by~\eqref{anisoGauss}. We stress that the sequence
$\bgamma=\{\gamma_\ell\}_{\ell=1}^{\infty}$ of shape parameters
can be arbitrary. In particular, we may consider
the isotropic case for which all $\gamma_\ell=\gamma>0$.

We want to verify how fast the minimal errors
$e^{\rm wor-all}(n,H_d)$
and $e^{\rm wor-std}(n,H_d)$
go to zero, and what the rate of convergence of these sequences is,
see~\eqref{rate}.
\begin{theorem}\label{thm1}
$$
r^{\rm wor-all}(H_d)
=r^{\rm wor-std}(H_d)=\infty.
$$
\end{theorem}

\begin{proof}
For the class $\lall$ we know that
$e^{\rm wor-all}(n,H_d)=\sqrt{\lambda_{d,n+1}}$,
where $\lambda_{d,n+1}$ is the $(n+1)^{\text{st}}$ largest eigenvalue
of $W_d$. Lemma \ref{lem1} demonstrates that $\lambda_{d,n+1}$ is proportional to $(n+1)^{-1/\tau}$ times a dimension dependent constant.
This implies that
$r^{\rm wor-all}(H_d)\ge 1/(2\tau)$
and since $\tau$ can be arbitrarily small, we conclude that
$$
r^{\rm wor-all}(H_d)=\infty,
$$
as claimed.

Consider now the class $\lstd$.  We use Theorem 5
from~\cite{KWW09}, which states that if there exist numbers
$p>1$ and $B$ such that
\begin{equation}\label{assumption1}
\lambda_{d,n}\le B\,n^{-p}\ \ \ \mbox{for all}\ \ \ n\in\naturals
\end{equation}
then for all $\delta\in(0,1)$ and $n\in\naturals$ there exists a linear
algorithm $A_n$ that  uses at most~$n$ function values and its worst
case error is bounded by
$$
e^{\rm wor}(A_n)\le B\,C_{\delta,p}\,(n+1)^{-(1-\delta)\,p^2/(2p+2)}.
$$
Here, $C_{\delta,p}$ is independent of $n$ and $d$ and depends only on
$\delta$ and $p$.

Note that assumption~\eqref{assumption1} holds in our case for an
arbitrarily large $p$ with $B$ that can depend on $d$. Hence,
$r^{\rm wor-std}(H_d)\ge (1-\delta)\,p^2/(2p+2)$, and since
$\delta$ can be arbitrarily small and $p$ can be arbitrarily
large we conclude
$$
r^{\rm wor-std}(H_d)=\infty,
$$
as claimed. This completes the proof.
\end{proof}

We stress that the algorithm $A_n$ that was used in the proof
is non-constructive. However, there are known algorithms that
use only function values and whose worst case error goes to zero
like $n^{-p}$ for an arbitrary large $p$. In fact,
given a design, it is known that the spline algorithm is
the best way to use the function data given via that design.
Thus, the search for an algorithm with optimal convergence rates
focuses on the choice of a good design.  One such design was
proposed by Smolyak already in 1963, see \cite{Sm63}, and today
it is usually referred to as a sparse grid, see \cite{BunGrie04a}
for a survey.
An associated algorithm from which this design naturally arises
is Smolyak's algorithm. The essence of this algorithm
is to use a certain tensor product of univariate algorithms. Then,
if the univariate algorithm has the worst case error
of order $n^{-p}$, the worst case error for the $d$-variate
case is also of order $n^{-p}$ modulo some powers of $\ln\,n$, see
e.g.,~\cite{WW95}.

Theorem~\ref{thm1}
states that as long as only the rate of convergence is considered,
the function approximation problem for Hilbert spaces with Gaussian
kernels is easy. In fact, it is not surprising since functions of this
class are very smooth.
However, the rate of convergence tells us
nothing about the dependence on $d$. As long as $d$ is small the
dependence on $d$ is irrelevant. But if $d$ is large we want to check
the dependence on~$d$. We are especially afraid of an exponential
dependence on $d$ which is called after Bellman the~\emph{curse of
dimensionality}. It also may happen that we have a tradeoff
between the rate of convergence and dependence on $d$.
Furthermore, the results may now depend on the weights $\gamma_\ell$.
This is the subject of our next sections.

\section{Tractability for the Absolute Error Criterion} \label{abserrsec}

As in the previous section we consider the function approximation
problem for Hilbert spaces  $H_d=H(K_d)$ with a Gaussian kernel.
We now consider the absolute error criterion and we want to verify
whether polynomial tractability holds. Let us recall that
we study the minimal number of functionals from the class
$\lall$ or $\lstd$ needed to guarantee a worst case error of at most $\e$,
 $$
n^{\rm wor-abs-\stdall}(\e,H_d)=\min\left\{\,n\ |\ \
e^{\rm wor-\stdall}(n,H_d)\le \e\right\}, \qquad \stdall \in \{{\rm std},{\rm all}\}.
$$

\subsection{Arbitrary Linear Functionals}\quad

We first analyze the class $\lall$ and  polynomial
tractability.
\begin{theorem}\label{thm2}
Consider the function approximation problem $\mathcal{I}=\{I_d\}_{d \in \naturals}$
for Hilbert spaces with isotropic
or anisotropic Gaussian kernels with arbitrary positive $\gamma_\ell$
for the class $\lall$ and
the absolute error criterion. Then
\begin{itemize}
\item $\mathcal{I}$
is strongly polynomially tractable
with exponent of strong polynomial tractability at most $2$.
For all $d\in\naturals$ and $\e\in(0,1)$ we have
\begin{eqnarray*}
e^{\rm wor-all}(n,H_d)&\le&(n+1)^{-1/2},\\
n^{\rm wor-abs-all}(\e,H_d)&\le&\e^{-2}.
\end{eqnarray*}
\item
For the isotropic Gaussian kernel the exponent of
strong tractability is $2$, so that the bound above is best possible
in terms of the exponent of $\e^{-1}$. Furthermore
strong polynomial tractability is equivalent to polynomial
tractability.
\end{itemize}
\end{theorem}

\begin{proof}
We use Theorem 5.1 from~\cite{NW08}. This theorem says that
$\mathcal{I}$ is strongly polynomially tractable iff
there exist two positive numbers $C_1$ and $\tau$ such that
$$
C_2:=\sup_{d\in\naturals}
\left(\sum_{j=\lceil
    C_1\rceil}^\infty\lambda^\tau_{d,j}\right)^{1/\tau}<\infty.
$$
If so, then
$$
n^{\rm wor-abs-all}(\e,H_d)\le (C_1+C_2^\tau)\,\e^{-2\tau}
\ \ \
\mbox{for all}\ \ \ d\in\naturals \ \ \mbox{and}\ \ \e\in(0,1).
$$
Furthermore, the exponent of strong polynomial tractability is
$$
p^{\rm all}=\inf\{2\tau\,|\ \ \mbox{$\tau$\ for which
$C_2<\infty$}\}.
$$
Let $\tau=1$. Then, by \eqref{sumeigone} it follows that no matter what the weights $\gamma_\ell$ are, we can
take an arbitrarily small $C_1$ so that $\lceil C_1\rceil=1$ and
$C_2=1$ as well as
$n^{\rm wor-abs-all}(\e,H_d)\le (C_1+1)\,\e^{-2}$.
For $C_1$ tending to zero, we conclude the bound
$$
n^{\rm wor-abs-all}(\e,H_d)\le \,\e^{-2}.
$$
Furthermore, by \eqref{nplusfirst} in Lemma \ref{lem1} it follows that
$$
e^{\rm wor-all}(n,H_d)=\sqrt{\lambda_{d,n+1}}\le (n+1)^{-1/2},
$$
as claimed.

Assume now the isotropic case, i.e., $\gamma_\ell=\gamma$
for all $j\in\naturals$.  Then for any positive $C_1$ and $\tau$ we
use Lemma 1 and obtain
\begin{eqnarray*}
\sum_{j=\lceil C_1\rceil}^\infty\lambda_{d,j}^\tau&=&
\sum_{j=1}^\infty\lambda_{d,j}^\tau-
\sum_{j=1}^{\lceil C_1\rceil-1}\lambda_{d,j}^\tau
\\
&=&\left(\frac{(1-\omega_\gamma)^\tau}{1-\omega_\gamma^\tau}\right)^d-
\sum_{j=1}^{\lceil C_1\rceil-1}\lambda_{d,j}^\tau\\
&\ge&
\left(\frac{(1-\omega_\gamma)^\tau}{1-\omega_\gamma^\tau}\right)^d-
\left(\lceil C_1\rceil-1\right)\lambda_{d,1}^\tau\\
&=&
\left(\frac{(1-\omega_\gamma)^\tau}{1-\omega_\gamma^\tau}\right)^d-
\left(\lceil C_1\rceil-1\right)(1-\omega_\gamma)^{\tau\,d}.
\end{eqnarray*}

For $\tau\in(0,1)$, we know from Lemma \ref{lem1}
that $(1-\omega_\gamma)^\tau/(1-\omega_\gamma^\tau)>1$, and therefore
the last expression goes exponentially fast to infinity with $d$.
This proves that $C_2=\infty$ for all $\tau\in(0,1)$.
Hence, the exponent of strong tractability is two.

Finally, to prove that strong polynomial tractability is equivalent
to polynomial tractability, it is enough to show that
polynomial tractability implies strong polynomial tractability.
From Theorem 5.1 of~\cite{NW08} we know that
polynomial tractability holds iff
there exist numbers $C_1>0$, $q_1\ge0$, $q_2\ge0$ and $\tau>0$
such that
$$
C_2:=\sup_{d\in\naturals}\left\{d^{-q_2}
\left(\sum_{j=\lceil
    C_1\,d^{\,q_1}\rceil}^\infty\lambda^\tau_{d,j}\right)^{1/\tau}
\right\}<\infty.
$$
If so, then
$$
n^{\rm wor-abs-all}(\e,H_d)\le(C_1+C_2^\tau)\,d^{\,\max(q_1,q_2\tau)}\,
\e^{-2\tau}
$$
for all $\e\in(0,1)$ and $d\in\naturals$. Note that for all $d$ we
have
$$
d^{-q_2\tau}\left(\frac{(1-\omega_\gamma)^\tau}{1-\omega_\gamma^\tau}\right)^d
-d^{-q_2\tau}
\left(\lceil C_1\rceil-1\right)(1-\omega_\gamma)^{\tau\,d}
\le C_2^\tau<\infty.
$$

This implies that $\tau\ge1$. On the other hand, for $\tau=1$ we can
take $q_1=q_2=0$ and arbitrarily small $C_1$, and
obtain strong tractability.
This completes the proof.
\end{proof}

We now compare Theorems~\ref{thm1} and~\ref{thm2}. Theorem~\ref{thm1} says
that for any $p$ we have
$$
e^{\rm wor-all}(n,H_d)=\mathcal{O}(n^{-p})
$$
but the factor in the big $\mathcal{O}$ notation may depend on $d$.
In fact, from Theorem~\ref{thm2} we conclude that, indeed, for the isotropic
case it depends
more than polynomially on $d$ for all $p>1/2$. Hence, the good rate of
convergence does not necessarily mean much for large $d$.

The exponent of strong polynomial tractability is $2$ for the isotropic
case. We now check how the exponent of strong polynomial tractability
depends on the sequence
$\bgamma=\{\gamma_\ell\}_{\ell \in \naturals}$ of shape parameters.
The determining factor is the quantity $r(\bgamma)$ introduced
in \eqref{wgammaform},  which measures the rate of decay of
the shape parameter sequence.

\begin{theorem}\label{thm3}
Consider the function approximation problem $\mathcal{I}=\{I_d\}_{d \in \naturals}$
for Hilbert spaces with isotropic or anisotropic
Gaussian kernels for the class $\lall$ and
the absolute error criterion. Let $r(\bgamma)$ be the rate of
decay of shape parameters. Then
\begin{itemize}
\item $\mathcal{I}$
is strongly polynomially tractable
with exponent
$$
p^{\rm all}= \min\left(2,\frac1{r(\bgamma)}\right)\le2.
$$
\item
For all $d\in\naturals$, $\e\in(0,1)$ and $\delta\in(0,1)$ we have
\begin{eqnarray*}
e^{\rm wor-all}(n,H_d)&=&\mathcal{O}\left(n^{-1/p^{\rm all}+\delta}
\right) = \mathcal{O}\left(n^{-\max(r(\bgamma),1/2)+\delta}
\right),\\
n^{\rm wor-abs-all}(\e,H_d)&=&
\mathcal{O}\left(\e^{-(p^{\rm all}+\delta)}\right),
\end{eqnarray*}
where the factors in the big $\mathcal{O}$ notation do not depend on
$d$ and $\e^{-1}$ but may depend on $\delta$.
\item
Furthermore, in the case of ordered shape parameters, i.e., $\gamma_1 \ge \gamma_2 \ge \cdots$ if
$$
n^{\rm wor-abs-all}(\e,H_d)=\mathcal{O}\left(\e^{-p}\,d^{\,q}\right)
\ \ \ \mbox{for all}\ \ \e\in(0,1)\ \mbox{and}\ d\in\naturals,
$$
then $p\ge p^{\rm all}$,
which means that strong polynomial tractability is
equivalent to polynomial tractability.
\end{itemize}
\end{theorem}

\begin{proof}  As in the proof of Theorem \ref{thm2},
$\mathcal{I}$ is strongly polynomially tractable iff
there exist two positive numbers $C_1$ and $\tau$ such that
$$
C_2:=\sup_{d\in\naturals}
\left(\sum_{j=\lceil
    C_1\rceil}^\infty\lambda^\tau_{d,j}\right)^{1/\tau}<\infty.
$$
Furthermore, the exponent $p^{\rm all}$ of strong polynomial
tractability is the infimum of $2\tau$ for which this
condition holds. Proceeding similarly as before, we have
$$
\sum_{j=\lceil
    C_1\rceil}^\infty\lambda^\tau_{d,j}\le
\sum_{j=1}^\infty\lambda^\tau_{d,j}=\prod_{\ell=1}^\infty
\frac{(1-\omega_{\gamma_\ell})^\tau}{1-\omega_{\gamma_\ell}^\tau}
$$
and since $\lambda_{d,j}<1$
$$
\sum_{j=\lceil
    C_1\rceil}^\infty\lambda^\tau_{d,j}\ge
\sum_{j=1}^\infty\lambda^\tau_{d,j}-C_1=
\prod_{\ell=1}^\infty\frac{(1-\omega_{\gamma_\ell})^\tau}
{1-\omega_{\gamma_\ell}^\tau}-C_1.
$$
Therefore, $\mathcal{I}$ is strongly polynomially tractable iff
there exists a positive $\tau$ such that
$$
C_3:=\prod_{\ell=1}^\infty\frac{1-\omega_{\gamma_\ell}}
{(1-\omega_{\gamma_\ell}^\tau)^{1/\tau}}<\infty,
$$
and the exponent $p^{\rm all}$ is the infimum of $2\tau$ for which the
last condition holds.

As we already know, this holds for
$\tau=1$. Take now $\tau\in(0,1)$.
Since
$(1-\omega_{\gamma_\ell})/(1-\omega_{\gamma_\ell}^\tau)^{1/\tau}>1$
then $C_3<\infty$ implies that
$$
\lim_{\ell\to\infty}\frac{1-\omega_{\gamma_\ell}}
{(1-\omega_{\gamma_\ell}^\tau)^{1/\tau}}=1.
$$
Taking into account \eqref{omegadef},
it is easy to check that the last condition is equivalent to
$$
\lim_{\ell\to\infty}\omega_{\gamma_\ell}=\lim_{\ell\to\infty}\gamma_\ell^2=0.
$$
Furthermore, $C_3<\infty$ implies that
$$
\sum_{\ell=1}^\infty\gamma_\ell^{2\tau}<\infty,
$$
and $r(\bgamma)\ge 1/(2\tau)>1/2$. Hence, $p^{\rm all}<2$ only if
$r(\gamma)>1/2$. On the other hand, $2\tau\ge 1/r(\bgamma)$ and
therefore $p^{\rm all}\ge1/r(\bgamma)$. This
establishes the formula for $p^{\rm all}$.
The estimates on $e^{\rm wor-all}(n,H_d)$ and
$n^{\rm wor-abs-all}(\e,H_d)$ follow from the definition of strong
tractability.

Assume now polynomial tractability with $p<2$ and an arbitrary $q$.
Then $\lambda_{d,n+1}\le\e^2$ for $n=\mathcal{O}(\e^{-p}d^q)$. Hence,
$$
\lambda_{d,n+1}=\mathcal{O}(d^{\,2q/p}(n+1)^{-2/p}).
$$
This implies
$$
\prod_{j=1}^d\frac{(1-\omega_{\gamma_\ell})^\tau}
{1-\omega_{\gamma_\ell}^\tau}=
\sum_{\ell=1}^\infty\lambda_{d,\ell}^\tau=\mathcal{O}(d^{\,2q\tau/p})\ \ \
\mbox{for all}\ \ \ 2\tau>p.
$$
For $\tau<1$, this yields
$$
\exp\left(\sum_{\ell=1}^d\gamma_\ell^{2\tau}\right)=\mathcal{O}(d^{\,2q\tau/p}).
$$
Therefore
$$
\limsup_{\ell\to\infty}\frac{\sum_{\ell=1}^d\gamma_\ell^{2\tau}}{\ln\,d}<\infty.
$$
Since the $\gamma_\ell$'s are ordered, we have
$$
\frac{d\gamma_{d}^{2\tau}}{\ln\,d}\le \frac{\sum_{\ell=1}^d
\gamma_{\ell}^{2\tau}}{\ln\,d}=\mathcal{O}(1),
$$
and $\gamma_d=\mathcal{O}((\ln(d)/d)^{1/(2\tau)})$. Hence,
$r(\bgamma)\ge 1/(2\tau)$ and $r(\bgamma)\ge1/p$. This means that
$2>p\ge1/r(\bgamma)=p^{\rm all}$, as claimed.
\end{proof}

It is interesting to notice that the last part of Theorem 3 does not
hold, in general, for unordered shape parameters. Indeed, for $s>1/2$,
take
\begin{eqnarray*}
\gamma_{a_k}&=&1 \ \ \ \mbox{for all natural $k$ with}\ \ a_k=2^{2^k},\\
\gamma_{\ell}&=&\frac1{\ell^s}\ \ \ \mbox{for all natural
$\ell$ not equal to $a_k$}.
\end{eqnarray*}
Then strong polynomial tractability holds with the exponent $2$ since
$C_3=\infty$ in the proof of Theorem 3 for all $\tau<1$. On the other
hand, we have polynomial tractability with $p=1/s<2$ and
$q$ arbitrarily close to $1/(2s)$. Indeed, for $\tau=1/(2s)$ and
$q_1=0$ and $q_2>1$ we have
\begin{eqnarray*}
d^{-q_2}\sum_{\ell=1}^\infty\lambda_{d,\ell}^\tau&=&d^{-q_2}\prod_{\ell}^d
\frac{(1-\omega_{\gamma_\ell})^\tau}{1-\omega_{\gamma_\ell}}\\
&=&d^{-q_2}\left(\frac{1-\omega_1)^\tau}{1-\omega_1}\right)^{\mathcal{O}(1)+
\ln\,\ln\, d}\,\mathcal{O}(d)<\infty.
\end{eqnarray*}
This implies that
$$
n^{\rm wor-abs-all}(\e,H_d)=\mathcal{O}\left(d^{\,q_2/(2s)}\,\e^{-1/s}\right).
$$
\vskip 1pc
Theorem 3 states that the exponent of strong polynomial tractability is $2$
for all shape parameters for which $r(\bgamma)\le 1/2$. Only if
$r(\bgamma)>1/2$ is the exponent smaller than~$2$.
Again, although the rate of convergence of
$e^{\rm wor-all}(n,H_d)$ is always excellent, the
dependence on $d$ is eliminated only at the expense of the exponent
which must be roughly  $1/p^{\rm all}$.  Of course, if we take an
exponentially decaying sequence of shape parameters, say,
$\gamma_\ell=q^{\,\ell}$ for some
$q\in(0,1)$, then $r(\bgamma)=\infty$ and $p^{\rm all}=0$. In this
case, we have an excellent rate of convergence without any dependence on $d$.

Although Theorem~\ref{thm2} is for Gaussian kernels, it is easy to
extend this theorem for other positive definite translation invariant
or radially symmetric kernels.
Indeed, for translation invariant kernels the only difference is that
for $\tau=1$ the sum of the eigenvalues is not necessarily one but
$$
\sum_{j=1}^\infty\lambda_{d,j}=\tK_d(\bzero).
$$
Hence, for all $\e\in(0,1)$ and $d\in\naturals$ we have
$$
e^{\rm wor-all}(n,H_d)\le\left[\frac{\tK_d(\bzero)}{n+1}\right]^{1/2}\ \
\mbox{and}\ \
n^{\rm wor-abs-all}(n,H_d)\le \tK_d(\bzero)\,\e^{-2}.
$$
Tractability then depends on how $\tK_d(\bzero)$ depends on $d$. In
particular, it is easy to check the following facts.
\begin{itemize}
\item If
$$
\sup_{d\in\naturals}\tK_d(\bzero)<\infty
$$
then we have
strong polynomial tractability with exponent at
most $2$, i.e.,
$$
n^{\rm wor-all}(n,H_d)=\mathcal{O}\left(\e^{-2}\right).
$$
\item If there exists a nonnegative $q$ such that
$$
\sup_{d\in\naturals}\tK_d(\bzero)\,d^{-q}<\infty
$$
then we have polynomial tractability and
$$
n^{\rm wor-all}(n,H_d)=\mathcal{O}\left(d^{\,q}\,\e^{-2}\right).
$$
\item If
$$
\lim_{d\to\infty}\frac{\ln\,\max(\tK_d(\bzero),1)}{d}=0
$$
then we have weak tractability.
\end{itemize}

For radially symmetric kernels, the situation is even simpler since
$$
\sum_{j=1}^\infty\lambda_{d,j}=\kappa(0),
$$
and it does not depend on $d$. Hence,
$$
e^{\rm wor-all}(n,H_d)\le\left[\frac{\kappa(0)}{n+1}\right]^{1/2}\ \
\mbox{and}\ \
n^{\rm wor-abs-all}(n,H_d)\le \kappa(0)\,\e^{-2},
$$
and we have strong polynomial tractability with exponent at most $2$.

Extending Theorem~\ref{thm3} to arbitrary stationary or
isotropic kernels is not so straightforward.
To achieve smaller
strong tractability exponents than $2$
one needs to know the sum of the powers of eigenvalues,
and their dependence on $d$.  One would suspect,
as is the case for Gaussian kernels,
that some sort of anisotropy is needed to
obtain better strong tractability exponents than $2$.

\subsection{Only Function Values}\quad

We now turn to the class $\lstd$ and prove the following theorem.
\begin{theorem}\label{thm4}
Consider the function approximation problem
$\mathcal{I}=\{I_d\}_{d \in \naturals}$
for Hilbert spaces with isotropic
or anisotropic Gaussian kernels for the class $\lstd$ and
the absolute error criterion. Then
\begin{itemize}
\item $\mathcal{I}$
is strongly polynomially tractable
with exponent of strong polynomial tractability at most $4$.
For all $d\in\naturals$ and $\e\in(0,1)$ we have
\begin{eqnarray*}
e^{\rm wor-std}(n,H_d)&\le&\frac{\sqrt{2}}{n^{1/4}}
\,\left(1+\frac1{2\sqrt{n}}\right)^{1/2},\\
n^{\rm wor-abs-std}(\e,H_d)
&\le& \left\lceil\frac{(1+\sqrt{1+\e^2})^2}{\e^{4}}\right\rceil.
\end{eqnarray*}
\item
For the isotropic Gaussian kernel the exponent of
strong tractability is at least $2$.
Furthermore
strong polynomial tractability is equivalent to polynomial
tractability.
\end{itemize}
\end{theorem}

\begin{proof}
We now use Theorem 1 from~\cite{WW01}. This theorem says that
\begin{equation}\label{useitagain}
e^{\rm wor-std}(n,H_d)\le
\min_{k=0,1,\dots}\left([e^{\rm
    wor-all}(k,H_d)]^2+ \frac{k}{n}\right)^{1/2}.
\end{equation}
Taking $k=\lceil n^{-1/2}\rceil$ and remembering that
$e^{\rm wor-all}(k,H_d)\le k^{-1/2}$ we obtain
$$
e^{\rm wor-std}(n,H_d)\le
\left(\frac1{\sqrt{n}}+\frac{1+\sqrt{n}}{n}\right)^{1/2}=
\frac{\sqrt{2}}{n^{1/4}}
\,\left(1+\frac1{2\sqrt{n}}\right)^{1/2},
$$
as claimed. Solving $e^{\rm wor-std}(n,H_d)\le\e$, we obtain
the bound on $n^{\rm wor-abs-std}(\e,H_d)$.

For the isotropic case, we know from Theorem~\ref{thm2}
that the exponent of strong
tractability for the class $\lall$ is $2$. For the class $\lstd$,
the exponent cannot be smaller.

Finally, assume that we have polynomial tractability for the class
$\lstd$. Then we also have polynomial tractability for the class $\lall$.
From Theorem~\ref{thm2} we know that then strong tractability for the class
$\lall$ holds. Furthermore we know that the exponent of strong
tractability is $2$ and $n^{\rm wor-abs-all}(\e,H_d)\le \e^{-2}$.
As above, we then get strong tractability also for $\lstd$ with the
exponent at most $4$.
This completes the proof.
\end{proof}

We do not know if the error bound of order $n^{-1/4}$ is sharp for the class
$\lstd$. We suspect that it is \emph{not} sharp and that
maybe even an error bound
of order $n^{-1/2}$ holds  for the class~$\lstd$ exactly as for
the class $\lall$.

For fast decaying shape parameters it is possible to improve
Theorem~\ref{thm4}. This is the subject of our next theorem.

\begin{theorem}\label{thm5}
Consider the function approximation problem $\mathcal{I}=\{I_d\}_{d \in \naturals}$
for Hilbert spaces with isotropic
or anisotropic Gaussian kernels for the class $\lstd$ and
the absolute error criterion.
Let $r(\bgamma)>1/2$. Then
\begin{itemize}
\item $\mathcal{I}$
is strongly polynomially tractable
with exponent at most
$$
p^{\rm std}=\frac1{r(\bgamma)}+\frac1{2\,r^2(\bgamma)}=p^{\rm
  all}+\tfrac12\left[p^{\rm all}\right]^2<4.
$$
\item
For all $d\in\naturals$, $\e\in(0,1)$ and $\delta\in(0,1)$ we have
\begin{eqnarray*}
e^{\rm wor-std}(n,H_d)&=&\mathcal{O}\left(n^{-1/p^{\rm std}+\delta}\right)=\mathcal{O}\left(n^{-r(\bgamma)/[1+1/(2r(\bgamma))]+\delta}\right),\\
n^{\rm wor-abs-std}(\e,H_d)
&=& \mathcal{O}\left(\e^{-(p^{\rm std}+\delta)}\right),
\end{eqnarray*}
where the factors in the big $\mathcal{O}$ notation do not depend on
$d$ and $\e^{-1}$ but may depend on $\delta$.
\end{itemize}
\end{theorem}
\begin{proof}
For $r(\bgamma)>1/2$, Theorem~\ref{thm3} for the class $\lall$
states that the exponent of strong polynomial tractability is
$p^{\rm all}=1/r(\bgamma)$. This means that for all $\eta\in(0,1)$
we have
$$
\lambda_{d,n}=\mathcal{O}(n^{-2r(\bgamma)+\eta}),
$$
with the factor in the big $\mathcal{O}$ notation independent of $n$
and $d$ but dependent on $\delta$. Since $2r(\bgamma)>1$,
it follows that for all positive $\eta$ small enough, $p=2r(\bgamma)-\eta>1$. Applying Theorem 5 from~\cite{KWW09} as in the proof of
Theorem~\ref{thm1}, it follows that for any $\delta_1\in(0,1)$ we have
\begin{eqnarray*}
e^{\rm wor-std}(n,H_d)&=&
\mathcal{O}\left(n^{-(1-\delta_1)p^2/(2p+2)}\right)=
\mathcal{O}\left(n^{-(1-\delta_1)(1+\mathcal{O}(\eta))
2w^2(\bgamma)/(2r(\bgamma)+1)}\right)\\
&=&
\mathcal{O}\left(n^{-1/(p^{\rm
      std}+\delta)}\right),
\end{eqnarray*}
again with the factor in the big $\mathcal{O}$ notation independent
of $n$ and $d$ but dependent on $\delta$. This leads to the estimates of
the theorem.
\end{proof}

Note that for large $r(\bgamma)$, the exponents of strong polynomial
tractability are nearly the same for both classes $\lall$ and $\lstd$.
For an exponentially decaying sequence of shape parameters, say, $\gamma_\ell=q^{\,\ell}$
for some $q\in(0,1)$, we have $p^{\rm all}=p^{\rm std}=0$, and
the rates of convergence are excellent and independent of $d$.

\section{Tractability for the Normalized Error Criterion} \label{normerrsec}

We now consider the function approximation problem for Hilbert spaces
$H_d(K_d)$ with a Gaussian kernel for the normalized error
criterion. That is, we want to find the smallest $n$ for which
$$
e^{\rm wor-\stdall}(n,H_d)\le \e\,\|I_d\|, \qquad \stdall \in \{{\rm std},{\rm all}\}.
$$
Note that $\|I_d\|=\sqrt{\lambda_{d,1}}\le1$ and it can be
exponentially small in $d$. Therefore
the normalized error criterion may be much harder
than the absolute error criterion and this is the reason
for a number of negative results for this error criterion.
It turns out that the isotropic and anisotropic cases are quite
different and we will study them in separate subsections. We begin with the case where the data are generated by arbitrary linear functionals. The class $\lstd$ is partially covered at the end.

\subsection{Isotropic Case with Arbitrary Linear Functionals}\quad

For the isotropic case, $\gamma_\ell=\gamma>0$, we have
$$
\|I_d\|=\tlambda_{\gamma,1}^{d/2} = (1 - \omega_{\gamma})^{d/2},
$$
and since $\tlambda_{\gamma,1}=1 - \omega_{\gamma}<1$, the norm of $I_d$ is exponentially
small. We are ready to present the following theorem.
\begin{theorem} \label{thm6}
Consider the function approximation problem $\mathcal{I}=\{I_d\}_{d \in \naturals}$
for Hilbert spaces with isotropic
Gaussian kernels for the class $\lall$ and for
the normalized error criterion. Then
\begin{itemize}
\item
$\mathcal{I}$
is not polynomially tractable,
\item
$\mathcal{I}$ is quasi-polynomially tractable
with exponent
$$
t^{\,\rm all}=t^{\,\rm all}(\gamma)=
\frac2{\ln\,
\frac{1+2\gamma^2+\sqrt{1+4\gamma^2}}{2\gamma^2}}.
$$
That is, for all $d\in\naturals$, $\e\in(0,1)$ and $\delta\in(0,1)$ we
have
\begin{eqnarray*}
e^{\rm wor-all}(n,H_d)&=&\mathcal{O}\left( \|I_d\|
\left(\frac1n\right)^{\frac1{(t^{\rm all}+\delta)\,(1+\ln\,d)}}
\,\left(\frac1{
\tfrac12(1+\sqrt{1+4\gamma^2})+\gamma^2}\right)^{d/4}\right),\\
n^{\rm wor-nor-all}(\e,H_d)&=&
\mathcal{O}\left(\exp\left((t^{\rm
all}+\delta)(1+\ln\,d)(1+\ln\,\e^{-1})\right)\right),
\end{eqnarray*}
where the factors in the big $\mathcal{O}$ notations are independent
of $n,\e^{-1}$ and $d$ but may depend on $\delta$.
\end{itemize}
\end{theorem}
\begin{proof}
The lack of polynomial tractability
follows, in particular,
from Theorem~5.6 of~\cite{NW08}. In fact, the lack of polynomial
tractability for the class $\lall$
holds for all tensor product problems with two positive
eigenvalues for the univariate case.

For quasi-polynomial tractability we use Theorem~3.3 of~\cite{GW10},
which states that quasi-polynomial tractability for the class $\lall$
holds for tensor product problems iff
the rate
$$
r=\sup\left\{\,\beta\ge0\,|\ \
  \lim_{n\to\infty}\tlambda_{\gamma,n}\,n^{\beta}=0\right\}
$$
of the univariate eigenvalues is positive and
the second largest univariate eigenvalue $\tlambda_{\gamma,2}$ is
strictly less than the largest univariate eigenvalue $\tlambda_{\gamma,1}$.
If so, then
the exponent of quasi-polynomial tractability is
$$
t^{\,\rm all}=\max\left(\frac2{r},\frac2{\ln\
    \tlambda_{\gamma,1}/\tlambda_{\gamma,2}}\right).
$$
In our case, $r=\infty$ and
$$
t^{\rm all}=\frac{2}{\ln \, \tlambda_{\gamma,1}/\tlambda_{\gamma,2}}
=\frac{2}{- \ln \, \omega_{\gamma} }
=\frac{2}{\ln \, \frac{1+2\gamma^2+\sqrt{1+4\gamma^2}}{2\gamma^2}}.
$$
The estimates of
$e^{\rm wor-all}(n,H_d)$ and $n^{\rm wor-nor-all}(\e,H_d)$
follow from the definition of quasi-polynomial tractability.
This completes the proof.
\end{proof}

For the isotropic case
we lose polynomial
tractability for the normalized error criterion
although even strong polynomial tractability
is present for the absolute error criterion.
This shows qualitatively that the normalized error criterion is much
harder. In this case we only have quasi-polynomial tractability.
Observe that the exponent of quasi-polynomial tractability depends on
$\gamma$ and we have
$$
\lim_{\gamma\to0}\,t^{\rm all}(\gamma)=0\ \ \ \mbox{and}
\lim_{\gamma\to\infty}\,t^{\rm all}(\gamma)=\infty.
$$
For some specific values of $\gamma$ we have
\begin{eqnarray*}
t^{\rm all}(2^{-1/2})&=&1.5186\dots,\\
t^{\rm all}(1)&=&2.0780\dots,\\
t^{\rm all}(2^{1/2})&=&2.8853\dots.
\end{eqnarray*}

\subsection{Anisotropic Case with Arbitrary Linear Functionals}\quad

We now consider the sequence $\{\gamma_\ell\}_{\ell
\in \naturals}$ of shape parameters and
ask when we can guarantee strong polynomial
tractability. As we shall
see, this holds for the class $\lall$
if $r(\bgamma)>0$ although the exponent of strong polynomial
tractability
is large for small $r(\bgamma)$. More precisely, we have the following
theorem, which is similar to Theorem \ref{thm3}.
\begin{theorem}\label{thm7}
Consider the function approximation problem $\mathcal{I}=\{I_d\}_{d \in \naturals}$
for Hilbert spaces with anisotropic
Gaussian kernels for the class $\lall$ and for
the normalized error criterion. Then
\begin{itemize}
\item
$\mathcal{I}$
is strongly polynomially tractable if $r(\bgamma)>0$.
If so, then the exponent is
$$
p^{\rm all}=\frac1{r(\bgamma)}.
$$
\item
Let $r(\bgamma)>0$.
Then for all $d\in\naturals$, $\e\in(0,1)$ and $\delta\in(0,1)$ we
have
\begin{eqnarray*}
e^{\rm wor-all}(n,H_d)&=&\mathcal{O}\left( \|I_d\|
n^{-1/p^{\rm all}+\delta}\right)=\mathcal{O}\left(
n^{-r(\bgamma)+\delta}\right),\\
n^{\rm wor-nor-all}(\e,H_d)&=&
\mathcal{O}\left(\e^{-(p^{\rm all}+\delta)}\right),
\end{eqnarray*}
where the factors in the big $\mathcal{O}$ notations are independent
of $n,\e^{-1}$ and $d$ but may depend on $\delta$.
\item
Furthermore, in the case of ordered shape parameters,
i.e., $\gamma_1 \ge \gamma_2 \ge \cdots$ if
$$
n^{\rm wor-nor-all}(\e,H_d)=\mathcal{O}\left(\e^{-p}\,d^{\,q}\right)
\ \ \ \mbox{for all}\ \ \e\in(0,1)\ \mbox{and}\ d\in\naturals,
$$
then $p\ge p^{\rm all}= \frac1{r(\bgamma)}$,
which means that strong polynomial tractability is
equivalent to polynomial tractability.
\end{itemize}
\end{theorem}
\begin{proof}
Theorem 5.2 of~\cite{NW08} states that strong polynomial tractability
holds iff there exits a positive number $\tau$ such that
$$
\tC_2:=\sup_{d}
\sum_{j=1}^\infty
\left(\frac{\lambda_{d,j}}{\lambda_{d,1}}\right)^\tau
=\prod_{\ell=1}^\infty\frac1{1-\omega_{\gamma_\ell}^\tau}<\infty.
$$
If so, then $n^{\rm wor-nor-all}(\e,H_d)\le \tC_2\,\e^{-2\tau}$
 for all $\e\in(0,1)$ and $d\in\naturals$,
and the exponent of strong polynomial tractability is the infimum of $2\tau$
for which $\tC_2<\infty$.

Clearly, $\tC_2<\infty$ iff
$$
\sum_{\ell=1}^\infty\omega_{\gamma_\ell}^\tau<\infty
\ \  \ \mbox{iff}\ \ \
\sum_{\ell=1}^\infty\gamma_\ell^{2\tau}<\infty.
$$
This holds iff $r(\bgamma)\ge 1/(2\tau)>0$. This also proves that
$p^{\rm all}=1/r(\bgamma)$.
The estimates on $e^{\rm wor-all}(n,H_d)$
and $n^{\rm wor-nor-all}(\e,H_d)$ follow from the definition of strong
tractability.

The case of polynomial tractability for ordered shape parameters
follows analogously to the proof in Theorem \ref{thm3}.
From Theorem 5.2 of \cite{NW08}, we know that
the problem is polynomially tractable with
$n^{\rm
  wor-nor-all}(\e,H_d)=\mathcal{O}\left(\e^{-2\tau}\,d^{\,q_2\tau}
\right)$ iff
$$
\tC_2:=\sup_{d\in\naturals}\,d^{-q_2}
\bigg[\sum_{j=1}^\infty
\left(\frac{\lambda_{d,j}}{\lambda_{d,1}}\right)^\tau \bigg]^{1/\tau}
=d^{-q_2}\prod_{\ell=1}^d\frac1{(1-\omega_{\ell}^\tau)^{1/\tau}}
<\infty.
$$
Proceeding as in the proof of Theorem 3, this can happen
for ordered shape parameters only if $\tau\ge 1/(2r(\bgamma))$.
Therefore, $p\ge p^{\rm all}= 1/r(\bgamma)$, as claimed.
\end{proof}

The essence of Theorem \ref{thm7} is that
under the normalized error criterion strong polynomial
and polynomial tractability for the class $\lall$
requires that the shape parameters tend to zero
polynomially fast so that $r(\bgamma)>0$. This condition
is stronger than what is required for the absolute error criterion.

It is interesting to compare strong polynomial tractability for the
absolute and normalized error criteria for the class $\lall$,
see Theorems~\ref{thm3}
and~\ref{thm7}. This is the subject of the next corollary.
\begin{corollary}\label{cor1}
Consider the function approximation problem $\mathcal{I}=\{I_d\}_{d \in \naturals}$
for Hilbert spaces with isotropic or anisotropic
Gaussian kernels for the class $\lall$. Let $r(\bgamma)$ be the rate of
convergence of shape parameters.
\begin{itemize}
\item Absolute error criterion:

$\mathcal{I}$
is always strongly polynomially tractable
with exponent
$$
p^{\rm all}=\min\left(2,\frac1{r(\bgamma)}\right)\le2.
$$
\item Normalized error criterion:

$\mathcal{I}$
is strongly polynomially tractable iff $r(\bgamma)>0$. If so,
the exponent is
$$
p^{\rm all}=\frac1{r(\bgamma)}.
$$
\end{itemize}
The strong tractability exponents under the two error
criteria are the same provided that $r(\bgamma)\ge 1/2$.
\end{corollary}

\subsection{Only Function Values}\quad

We now turn to the class $\lstd$. We do not know
if quasi-polynomial tractability holds for the class
$\lstd$ in the isotropic case. The theorems that we used
for the absolute error criterion
are not enough for the normalized error criterion. Indeed,
no matter how a positive
$k$ is defined in~\eqref{useitagain}
we must take $n$ exponentially large in $d$ if we want
to guarantee that the error is less than $\e\|I_d\|$. Similarly,
if we use~\eqref{assumption1} then we must guarantee that $p>1$ and
this makes the number $B$ exponentially large in $d$.
We leave as an open problem whether
quasi-polynomial tractability holds for the class $\lstd$.

We now discuss the initial error for
$\lim_{\ell\to\infty}\gamma_\ell=0$.
We have
$$
\|I_d\|=\prod_{\ell=1}^d\left(1-\omega_{\gamma_\ell}\right)^{1/2}
=\exp\left(\mathcal{O}(1)-\tfrac12\sum_{\ell=1}^d\gamma_\ell^2\right).
$$
For $r(\bgamma)\in[0,1/2)$, the initial error still goes exponentially
fast to zero, whereas for $r(\bgamma)=1/2$ it may go to zero or
be uniformly bounded from below by a positive number,
and finally for $r(\bgamma)>1/2$ it is always uniformly
bounded from below by a positive number. For example, take
$\gamma_\ell=\ell^{-\alpha}\ln^\beta(1+\ell)$ for a positive $\alpha$ and
real $\beta$. Then $r(\bgamma)=\alpha$. For $\alpha=\tfrac12$, the
initial error goes to zero for $\beta> -\tfrac12$, and is
of order~$1$ if $\beta\le-\tfrac12$.

This discussion shows that for $r(\bgamma)>1/2$ there is really no
difference between the absolute and normalized error criteria. This
means that for $r(\bgamma)>1/2$ we can apply Theorem~\ref{thm5} for the
class $\lstd$ with $\e$ replaced by $\e\|I_d\|=\Theta(\e)$.
For $r(\bgamma)=1/2$, Theorem~\ref{thm4} can be applied
if we assume additionally that
$\sum_{\ell=1}^\infty\gamma_\ell^2<\infty$. The last assumption
implies that $\|I_d\|=\Theta(1)$. We summarize this
discussion in the following corollary.

\begin{corollary} \label{cor2}
Consider the function approximation problem $\mathcal{I}=\{I_d\}_{d \in \naturals}$
for Hilbert spaces with anisotropic
Gaussian kernels for the class $\lstd$ and for
the normalized error criterion. Assume that
$$
r(\bgamma)>\tfrac12\ \ \mbox{or}\ \ \
\left(\ r(\bgamma)=\tfrac12\ \mbox{and}\
\sum_{\ell=1}^\infty\gamma_\ell^2<\infty\ \right).
$$
Then
\begin{itemize}
\item
$\mathcal{I}$
is strongly polynomially tractable with exponent at most
$$
p^{\rm std}=\frac1{r(\bgamma)}+\frac1{2\,r^2(\bgamma)}=p^{\rm all}+\tfrac12
\,\left[p^{\rm all}\right]^2\le 4.
$$
\item
For all $d\in\naturals$, $\e\in(0,1)$ and $\delta\in(0,1)$ we
have
\begin{eqnarray*}
e^{\rm wor-all}(n,H_d)&=&\mathcal{O}\left(
n^{-1/(p^{\rm all}+\delta)}\right),\\
n^{\rm wor-nor-all}(\e,H_d)&=&
\mathcal{O}\left(\e^{-(p^{\rm all}+\delta)}\right),
\end{eqnarray*}
where the factors in the big $\mathcal{O}$ notations are independent
of $n,\e^{-1}$ and $d$ but may depend on $\delta$.
\end{itemize}
\end{corollary}

The case $r(\bgamma)<1/2$
is open.  We do not know if polynomial tractability holds for the
class $\lstd$ in this case.

\end{document}